\documentclass{amsart}
\usepackage{amssymb}
\usepackage{amsmath}
\usepackage{latexsym}
\usepackage{graphicx}
\usepackage{amscd}
\usepackage{mathrsfs}
\usepackage[english]{babel}
\usepackage{verbatim}

\newtheorem{theorem}{Theorem}[section]

\newtheorem{cor}[theorem]{Corollary}
\newtheorem{conj}[theorem]{Conjecture}

\theoremstyle{definition}
\newtheorem{definition}[theorem]{Definition}
\newtheorem{example}[theorem]{Example}

\theoremstyle{remark}
\newtheorem{remark}[theorem]{Remark}

\numberwithin{equation}{section}

\newcommand{\llbr}{[\negthinspace[}
\newcommand{\rrbr}{]\negthinspace]}

\newcommand{\llpar}{(\negthinspace(}
\newcommand{\rrpar}{)\negthinspace)}

\newcommand{\LL}{\ensuremath{\mathbb{L}}}

\newcommand{\Z}{\ensuremath{\mathbb{Z}}}
\newcommand{\Q}{\ensuremath{\mathbb{Q}}}

\newcommand{\R}{\ensuremath{\mathbb{R}}}
\newcommand{\C}{\ensuremath{\mathbb{C}}}

\newcommand{\A}{\ensuremath{\mathbb{A}}}

\newcommand{\cX}{\ensuremath{\mathscr{X}}}

\newcommand{\cC}{\ensuremath{\mathscr{C}}}

\newcommand{\cY}{\ensuremath{\mathscr{Y}}}

\renewcommand{\R}{\ensuremath{\mathbb{R}}}
\renewcommand{\C}{\ensuremath{\mathbb{C}}}

\renewcommand{\A}{\ensuremath{\mathbb{A}}}

\renewcommand{\cY}{\ensuremath{\mathscr{Y}}}

\newcommand{\Spec}{\ensuremath{\mathrm{Spec}\,}}

\newcommand{\Var}{\mathrm{Var}}

\newcommand{\an}{\mathrm{an}}
\newcommand{\spe}{\mathrm{sp}}
\newcommand{\trop}{\mathrm{trop}}

\newcommand{\VF}{\mathrm{VF}}
\newcommand{\Vol}{\mathrm{Vol}}
\newcommand{\Hilb}{\mathrm{Hilb}}
\newcommand{\mot}{\mathrm{mot}}
\newcommand{\eu}{\mathrm{eu}}

\begin{document}
\title{Geometric invariants for non-archimedean semialgebraic sets}

\author[Johannes Nicaise]{Johannes Nicaise}
\address{ Imperial College,
Department of Mathematics, South Kensington Campus,
London SW7 2AZ, UK, and KU Leuven, Department of Mathematics, Celestijnenlaan 200B, 3001 Heverlee, Belgium}
\email{j.nicaise@imperial.ac.uk}

\thanks{The author is supported by the ERC Starting Grant MOTZETA (project 306610) of the European Research Council, and by long term structural funding (Methusalem
grant) of the Flemish Government.}

\begin{abstract}
This survey paper explains how one can attach geometric invariants to semialgebraic sets defined over
 non-archimedean fields, using the theory of motivic integration of Hrushovski and Kazhdan. It also discusses
 tropical methods to compute these invariants in concrete cases, as well as an application
  to refined curve counting, developed in collaboration with Sam Payne and Franziska Schroeter.
 \end{abstract}

\maketitle

\section{Introduction}
 Let $K$ be the field of complex Puiseux series: $$K=\bigcup_
 {n>0}\mathbb{C}\llpar t^{1/n}\rrpar.$$ This is an algebraic closure of the field of complex Laurent series $\mathbb{C}\llpar t\rrpar$.
 We denote by $v:K\to \Q\cup\{\infty\}$ the $t$-adic valuation.
  A {\em semialgebraic} subset of an algebraic $K$-variety $X$ is a subset of $X(K)$ that can locally be defined by finitely many Boolean operators and inequalities of the form $v(f)\leq v(g)$ where $f,g$ are algebraic functions on $X$.
    The aim of this survey paper is to explain how one can attach geometric invariants to semialgebraic sets over the field $K$ using the theory of motivic integration developed by Hrushovski and Kazhdan \cite{HK}. The motivation for this construction is twofold:
 \begin{enumerate}
 \item Semialgebraic sets occur naturally in tropical and non-archimedean geometry. For instance, given a family of subvarieties of an algebraic torus, the locus of fibers of the family with fixed tropicalization is semialgebraic (see Example \ref{exam:semialg}\eqref{it:examtrop}).

 \item Even if one is ultimately interested in computing invariants for algebraic varieties, it is often useful to know that one can
  compute these invariants on semialgebraic decompositions of the variety, for instance to obtain tropical formulas.
 \end{enumerate}
 Both motivations play an essential role in an ongoing project with Sam Payne and Franziska Schroeter \cite{NPS}, which aims to give a geometric interpretation of the refined tropical multiplicities of Block and G{\"ottsche} \cite{BG} and to obtain a tropical correspondence theorem for the refined curve counting invariants of G{\"o}ttsche and Shende \cite{GS}. We will explain the main ideas in Section \ref{sec:appli}.

 The central tool in our approach is the {\em motivic volume} defined by Hrushovski and Kazhdan. This is a morphism
$$\mathrm{Vol}:K_0(\mathrm{VF}_K)\to K_0(\mathrm{Var}_\C)$$ from the Grothendieck ring of semialgebraic sets over the valued field $K$ to the Grothendieck ring of algebraic varieties over the residue field $\C$. With the help of this morphism, one can extend all the classical motivic invariants in algebraic geometry to semialgebraic sets, by composing $\mathrm{Vol}$ with the motivic invariant on $K_0(\mathrm{Var}_\C)$. In particular, this allows us to define the Hodge-Deligne polynomial, the $\chi_{-y}$-genus and the Euler characteristic of a semialgebraic set.
 \if false
 A common feature of all the theories of motivic integration is that they try to understand the structure of semialgebraic objects over $K$ in terms of
data living over the residue field $\C$ (that is, complex algebraic varieties) and over the value group $\Q=|K^{\ast}|$ (polyhedra). This aim is realized in the theory of Hrushovski and Kazhdan by a complete description of the Grothendieck ring of semialgebraic sets $K_0(\mathrm{VF}_K)$ as a tensor product of certain Grothendieck rings of $k$-varieties and polyhedra, respectively. Hrushovski and Kazhdan show that $K_0(\mathrm{VF}_K)$ is generated by the classes of the following types of semialgebraic sets:
\begin{itemize}
\item inverse images of closed $\Q$-rational polyhedra $\Gamma$ in $\mathbb{R}^n$ under the tropicalization map $\mathrm{trop}:(K^{\ast})^n\to \mathbb{R}^n$;
\item tubes around subvarieties $X$ of the special fibers of smooth $R$-schemes $\mathcal{X}$ of finite type.
\end{itemize}
 Moreover, they express in a simple and elegant way all the relations that exist between these classes. The motivic volume
 $$\mathrm{Vol}:K_0(\mathrm{VF}_K)\to K_0(\mathrm{Var}_\C)$$ is fully characterized by its values on the generators above:
 \begin{itemize}
 \item $\mathrm{Vol}(\mathrm{trop}^{-1}(\Gamma))=[\mathbb{G}^n_{m,\C}]$ for every $n\geq 1$ and every closed $\Q$-rational polyhedron $\Gamma$ in $\mathbb{R}^n$;
 \item the volume of the tube around $X$ in $\mathcal{X}$ equals $[X]$.
 \end{itemize}
 The motivation for the first expression is that we can think of $\mathrm{trop}^{-1}(\Gamma)$ as a $(R^{\ast})^n$-torsor over $\Gamma$, and that the volume of $(R^{\ast})^n$ according to the second expression is $[\mathbb{G}^n_{m,\C}]$.
 \fi
  In many situations, these invariants of semialgebraic sets have a natural geometric meaning. For instance, one can deduce from work by Martin \cite{M} and Hrushovski-Loeser \cite{HL} that the Euler characteristic of a semialgebraic set coincides with the one obtained from Berkovich's theory of {\'e}tale cohomology for $K$-analytic spaces. Moreover, when $X$ is an algebraic variety over $K$, the Hodge-Deligne polynomial of $X(K)$ (viewed as a semialgebraic set) equals the Hodge-Deligne polynomial of the limit mixed Hodge structure associated with $X$.

 In order to compute these motivic invariants in concrete examples, we have established a tropical expression for the class in $K_0(\mathrm{VF}_K)$ of
   a sch{\"o}n subvariety $X$ of an algebraic $K$-torus $\mathbb{G}^n_{m,K}$. The sch{\"o}nness assumption is a generic non-degeneracy condition that is often used in tropical geometry; see Section \ref{ss:tropical} for a precise definition. The tropicalization  of $X$ is the closure in $\R^n$ of the image of $X(K)$ under the {\em tropicalization map}  $$\trop:(K^{\ast})^n\to \Q^n:(x_1,\ldots,x_n)\mapsto (v(x_1),\ldots,v(x_n)).$$
      Every rational polyhedral decomposition $\Sigma$ of the tropicalization of $X$ gives rise to a semialgebraic decomposition of $X(K)$ whose pieces are the inverse images of the open cells of $\Sigma$ under the tropicalization map. This leads to an explicit expression for the class of $X(K)$ in $K_0(\mathrm{VF}_K)$ involving the cells of $\Sigma$ and the corresponding initial degenerations of $X$ (see Theorem \ref{thm:trop}).

 Although, for our purposes, the case where $K$ is the field of Puiseux series is sufficient, we will work in greater generality, since this does not require any additional efforts. Let $K$ be any algebraically closed real-valued field of equal characteristic zero. We denote by $R$, $k$ and $G\subset (\R,+)$ the valuation ring, residue field and value group of $K$, respectively, and by $v:K^{\ast}\to G$ the valuation map. If $K=\cup_{n>0} \C\llpar t^{1/n}\rrpar$ then $R=\cup_{n>0}\C\llbr t^{1/n}\rrbr$, $k=\C$ and $G=\Q$. In any case, our assumptions imply that $G$ is divisible and that $k$ is algebraically closed.
    We extend the valuation $v$ to $K$ by setting $v(0)=\infty$, and we extend the usual ordering on $\R$ to $\overline{\R}=\R\cup \{\infty\}$ by declaring that $a\leq \infty$ for all $a$ in $\overline{\R}$. For every positive integer $n$, we denote by $\trop$ the tropicalization map
 $$\trop:(K^{\ast})^n\to G^n\subset \R^n:(x_1,\ldots,x_n)\mapsto (v(x_1),\ldots,v(x_n)).$$

\section{The motivic volume of Hrushovski-Kazhdan}
 In this section, we will explain how the theory of motivic integration of Hrushovski-Kazhdan \cite{HK} allows us to associate a motivic volume to
 any semialgebraic set over $K$. The proofs in \cite{HK}  rely heavily on the model theory of algebraically closed valued fields. We have tried to present the main results in a more geometric fashion to make
 the theory accessible to algebraic geometers.

\subsection{Semialgebraic sets}
Let $X$ be a $K$-scheme of finite type. A subset $S$ of $X(K)$ is called a semialgebraic subset of $X$ if we can
 write it as a finite Boolean combination of sets of the form
 $$\{x\in U(K)\,|\,v(f(x))\geq v(g(x))\}$$ where $U$ is an open subscheme of $X$, and $f$ and $g$ are regular functions on $U$.
  The Cartesian product of two semialgebraic sets is again semialgebraic. If $f:X\to Y$ is a morphism of $K$-schemes of finite type and $S_Y$ is a semialgebraic subset of $Y$, then it is easy to see that $f^{-1}(S_Y)\cap X(K)$ is semialgebraic in $X$. Conversely, Robinson's  quantifier elimination theorem for algebraically closed valued fields \cite{rob} implies that, if  $S_X$ is a semialgebraic subset of $X$, then $f(S_X)\subset Y(K)$ is a semialgebraic subset of $Y$.

 \begin{example}\label{exam:semialg}\item
\begin{enumerate}
\item If $X$ is a $K$-scheme of finite type, then every constructible subset of $X(K)$ is semialgebraic. Indeed, locally on $X$, it is a finite Boolean combination of subsets of the form
    $$\{x\in X(K)\,|\,f(x)=0\}= \{x\in X(K)\,|\,v(f(x))\geq v(0)\},$$ with $f$ a regular function.

\item \label{it:spsemialg} Let $\cX$ be an $R$-scheme of finite type. The {\em specialization map}
    $$\spe_{\cX}:\cX(R)\to \cX(k)$$ is defined by reducing coordinates modulo the maximal ideal of $R$.
If $C$ is a constructible subset of $\cX(k)$, then the {\em tube} $\spe_{\cX}^{-1}(C)$ around $C$ in $\cX$ is a semialgebraic subset of $\cX_K$. To prove this, it suffices to consider the case where $\cX$ is affine and $C$ is closed in $\cX(k)$. If $t_1,\ldots,t_n$ generate the $R$-algebra $\mathcal{O}(\cX)$ and $C$ is the set of closed points of the zero locus of an ideal $(f_1,\ldots,f_\ell)$ in $\mathcal{O}(\cX)$, then
 $$\spe_{\cX}^{-1}(C)=\{x\in X(K)\,|\,v(t_i(x))\geq 0 \mbox{ and }v(f_j(x))>0\mbox{ for all }i,j\}.$$
  When $C$ is a constructible subset of $\cX_k$, rather than $\cX(k)$, we write $\spe_{\cX}^{-1}(C)$ for $\spe_{\cX}^{-1}(C\cap \cX(k))$.

\item  A $G$-rational polyhedron in $\R^n$ is an intersection of finitely many half-spaces of the form $$\{x\in G^n\,|\,a_1x_1+\ldots +a_nx_n\leq c\}$$ with $a_1,\ldots,a_n$ in $\Z$ and $c$ in $G$. It is clear from the definitions that,
if $\Gamma$ is a finite Boolean combination of $G$-rational polyhedra in $\R^n$, then $\trop^{-1}(\Gamma)$ is a semialgebraic subset of $K^n$.

\item \label{it:examtrop} A more sophisticated example is the following.
 Let $Y$ be a $K$-scheme of finite type and let $X$ be a subscheme of $Y\times_K \mathbb{G}_{m,K}^n$, for some $n>0$. We denote by $f:X(K)\to Y(K)$ the restriction of the projection morphism $Y\times_K \mathbb{G}_{m,K}^n\to Y$. Let $\Gamma$ be a finite Boolean combination of $G$-rational polyhedra in $\R^n$. Then the set of points $y$ in $Y(K)$ such that $\trop(f^{-1}(y))=\Gamma$ is a semialgebraic subset of $Y$, by Robinson's quantifier elimination theorem.
\end{enumerate}
  \end{example}

The above examples should make it clear that semialgebraic sets naturally arise in non-archimedean and tropical geometry. We will be particularly interested in  a special case of example \eqref{it:examtrop}: let $Y$ be a $K$-scheme of finite type and let $X(\Delta)$ be the toric surface over $K$ associated with a lattice polygon $\Delta$ in $\R^2$. Let $\mathscr{C}\to Y$ be a family of curves in $X(\Delta)$; thus $\cC$ is a closed subscheme of $X(\Delta)\times_K Y$ such that the fiber $\cC_y$ over each point $y$ of $Y(K)$ is a curve in $X(\Delta)$.
 Then the locus of points $y$ in $Y(K)$ such that $\mathscr{C}_y\cap (K^{\ast})^2$ tropicalizes to a fixed tropical curve $\Gamma$ is a semialgebraic subset of $Y$.

If $X,Y$ are $K$-schemes of finite type and $S_X\subset X$ and $S_Y\subset Y$ are semialgebraic sets, then a map $S_X\to S_Y$ is called semialgebraic if its graph is a semialgebraic subset of $X\times_K Y$. It is easy to check that the composition of two semialgebraic maps is again semialgebraic. Thus we can define a category $\VF_K$ whose objects are
pairs $(X,S)$ with $X$ a $K$-scheme of finite type and $S$ a semialgebraic set in $X$, and whose morphisms are semialgebraic maps. We will usually denote an object of $\VF_K$ simply by $S$, leaving the ambient variety $X$ implicit.

\subsection{Grothendieck rings}
The theory of Hrushovski and Kazhdan relies on a precise description of the Grothendieck ring $K_0(\VF_K)$ of semialgebraic sets, which we will now define.
 As an abelian group, $K_0(\VF_K)$ is defined by the following presentation.
 \begin{itemize}
 \item Generators: isomorphism classes $[S]$ of semialgebraic sets $S$ over $K$.
 \item Relations: if $X$ is a $K$-scheme of finite type and $T\subset S$ are semialgebraic subsets of $X$, then $[S]=[T]+[S\setminus T]$. These relations are often called {\em scissor relations}, because they allow to cut up a semialgebraic set into semialgebraic pieces.
 \end{itemize}
 We then obtain a ring structure on $K_0(\VF_K)$ by setting $[S]\cdot [S']=[S\times S']$ for all semialgebraic sets $S,S'$.

  The motivic volume will take its values in a different Grothendieck ring, namely, the Grothendieck ring $K_0(\Var_k)$ of varieties over the residue field $k$. It is defined similarly:
  as an abelian group, it is generated by the isomorphism classes $[X]$ of $k$-schemes of finite type $X$, subject to the relation $[X]=[Y]+[X\setminus Y]$ for every closed subscheme $Y$ of $X$. The ring structure is induced by the fiber product over $k$. It is customary to write $\LL$ for the class $[\A^1_k]$ of the affine line in $K_0(\Var_k)$.

\subsection{Definition of the motivic volume}
 A common aim of all the theories of motivic integration is to understand the structure of semialgebraic objects over $K$ in terms of
data living over the residue field $k$ (algebraic $k$-varieties) and over the value group $G$ (polyhedra). In the geometric approaches to motivic integration by Kontsevich, Denef--Loeser, Sebag and Loeser--Sebag, this is achieved by
analyzing the geometry of arc schemes or Greenberg schemes \cite{motint}. There are also approaches based on model theory: Cluckers and Loeser use cell decomposition to describe the shapes of semialgebraic sets and to define their motivic measure \cite{CL}. For us, the most convenient theory will be the one developed by Hrusohvski and Kazhdan in \cite{HK}, which is also based on the model theory of valued fields.
 It provides a complete description of the Grothendieck ring of semialgebraic sets $K_0(\mathrm{VF}_K)$ as a tensor product of certain graded Grothendieck rings of polyhedra and  $k$-varieties, respectively. This description is close in spirit to tropical geometry, where one decomposes subspaces of algebraic tori into polyhedra ({\em via} the tropicalization map) and so-called {\em initial degenerations} over the residue field. This analogy will be quite apparent in our tropical formula for the motivic volume in Theorem \ref{thm:trop}.

 There are two natural ways to produce semialgebraic sets over $K$ from objects over the residue field $k$ and the value group $G$. We start with the most elementary construction.
 Let $n$ be a positive integer and let $\Gamma$ be a finite Boolean combination of $G$-rational polyhedra in $\R^n$. We have seen in Example \ref{exam:semialg} that $\trop^{-1}(\Gamma)$ is a semialgebraic subset of $K^n$. Thus, we can consider its class
 $$\Theta(\Gamma,n):= [\trop^{-1}(\Gamma)]$$ in $K_0(\VF_K)$. It is elementary to see that this definition is invariant under affine transformations of $\R^n$ of the form
 $x\mapsto Ax+b$ with $A\in \mathrm{GL}_n(\Z)$ and $b\in G^n$, and that it is additive with respect to scissor operations on the polyhedron $\Gamma$.

 The second construction starts from a nonnegative integer $n$ and a $k$-scheme of finite type $X$ of dimension at most $n$.
  First, assume that $X$ is smooth over $k$.
 Then we can find a smooth $R$-scheme $\cX$ of relative dimension $n$ and  an immersion of $R$-schemes $X \to \cX$. The set $\spe_{\cX}^{-1}(X)$ is semialgebraic in $\cX_K$ by Example \ref{exam:semialg}, and we claim that the class
  $$\Theta(X,n):=[\spe_{\cX}^{-1}(X)]$$ in $K_0(\VF_K)$ does not depend on the choice of $\cX$. To see this, let $X\to \cY$ be an immersion into another smooth $R$-scheme of relative dimension $n$. Working locally on $X$ and using the scissor relations in $K_0(\VF_K)$, we may assume that there exists an \'etale morphism from
  $X$ onto a subscheme of $\A^n_k$ that extends to \'etale morphisms of $R$-schemes $\cX\to \A^n_R$ and $\cY\to \A^n_R$ (see \cite[18.1.1]{ega4.4}).
    If we view $X$ as a subscheme of $\cX_k\times_{\A^n_k}\cY_k$ via the diagonal embedding,
   then the fact that $R$ is henselian implies that the semialgebraic set $$\spe_{\cX\times_{\A^n_R}\cY}^{-1}(X)$$ is the graph of a bijection between $\spe_{\cX}^{-1}(X)$ and $\spe_{\cY}^{-1}(X)$. Hence, the semialgebraic sets $\spe_{\cX}^{-1}(X)$ and $\spe_{\cY}^{-1}(X)$ define the same class in $K_0(\VF_K)$.
       If $X$ is any $k$-scheme of finite type of dimension at most $n$, then
  we can write $X$ as a disjoint union of $k$-smooth subschemes $X_1,\ldots,X_r$. One checks easily that
   the element
  $$\Theta(X,n):=\sum_{i=1}^r \Theta(X_i,n)$$ in $K_0(\VF_K)$ does not depend on the choice of such a partition.

  These two constructions are not completely orthogonal, as is illustrated by the following examples.
 \begin{example}\label{exam:rel}\item
 \begin{enumerate}
 \item We consider the $0$-simplex $\Delta_0=\{0\}$ in $\R$. Then $\trop^{-1}(\Delta_0)$ is the semialgebraic subset $R^{\ast}$ of $K$.
 On the other hand, we can also write $R^{\ast}$ as $$\mathbb{G}_{m,R}(R)=\spe^{-1}_{\mathbb{G}_{m,R}}(\mathbb{G}_{m,k}).$$
  It follows that $[R^{\ast}]=\Theta(\Delta_0,1)=\Theta(\mathbb{G}_{m,k},1)$.

 \item Let $D$ be the open unit disk in $K$, that is, the set of all $x$ in $K$ such that $v(x)> 0$.
  Then we can write $D$ as the union of the point $\{0\}$ and the punctured open unit disk $$D\setminus \{0\}=\trop^{-1}(\R_{>0}),$$ which yields the expression
  $$[D]=\Theta(\Spec k,0)+\Theta(\R_{>0},1).$$ On the other hand, we can also view $D$ as $\spe^{-1}_{\A^1_R}(O)$, where $O$ denotes the origin of $\A^1_k$.
  Hence, we have $$\Theta(\Spec k,0)+\Theta(\R_{>0},1)=\Theta(\Spec k,1).$$
  \end{enumerate}
  \end{example}

In \cite{HK}, Hurshovski and Kazhdan have proven the following striking result.
 \begin{theorem}[Hrushovski-Kazhdan] \label{thm:HK}
  The $\Theta$-classes of $G$-rational polyhedra and $k$-schemes of finite type generate the Grothendieck ring of semialgebraic sets $K_0(\VF_K)$.
   Moreover, apart from the scissor relations for $G$-rational polyhedra and $k$-schemes of finite type, the relations described in Example \ref{exam:rel} are the only relations between the $\Theta$-classes in $K_0(\VF_K)$.
  \end{theorem}
Hrushovski and Kazhdan have formulated this result in a more precise way as an isomorphism between the ring $K_0(\VF_K)$ and a tensor product of certain graded Grothendieck rings of $k$-varieties and $G$-rational polyhedra.
 We are mostly interested in the following consequence of Theorem \ref{thm:HK}.

\begin{cor}\label{cor:HK}
There exists a unique ring morphism
$$\Vol:K_0(\VF_K)\to K_0(\Var_k)$$ with the following properties.
\begin{enumerate}
\item \label{it:defvol1} For every smooth $R$-scheme of finite type $\cX$ and every subscheme $X$ of $\cX_k$, we have $\Vol(\spe^{-1}_{\cX}(X))=[X]$.

\item \label{it:defvol2} If $\Gamma$ is a $G$-rational polyhedron in $\R^n$, then $\Vol(\trop^{-1}(\Gamma))=(\LL-1)^n$.
\end{enumerate}
\end{cor}
 \begin{proof}
 First, we show that the expression for $\Vol(\trop^{-1}(\Gamma))$ is compatible with the scissor relations for $G$-rational polyhedra.
 In fact, there exists a unique additive invariant $\chi'$ on the Boolean algebra generated by $G$-rational polyhedra in $\R^n$ that sends every $G$-rational polyhedron to $1$. This invariant can be expressed as $$\chi'(\Gamma)=\lim_{r\to +\infty}\chi_c(\Gamma\cap [-r,r]^n)$$
   for every finite Boolean combination $\Gamma$ of $G$-rational polyhedra, where $\chi_c$ denotes the singular Euler characteristic with compact supports (one can show that the limit stabilizes for sufficiently large $r$). Thus we can extend $\Vol$ to all $\Gamma$ in an additive way by setting $\Vol(\trop^{-1}(\Gamma))=\chi'(\Gamma)(\LL-1)^n$.
  By Theorem \ref{thm:HK}, it now only remains to observe that the expressions in \eqref{it:defvol1} and \eqref{it:defvol2} satisfy the relations described in Example \ref{exam:rel}, because $\chi'(\R_{>0})=0$.
 \end{proof}
\begin{remark}
A good way to think about the identity $\Vol(\trop^{-1}(\Gamma))=(\LL-1)^n$ for $G$-rational polyhedra $\Gamma$ in $\R^n$ is
 to view $\mathrm{trop}^{-1}(\Gamma)$ as a $(R^{\ast})^n$-torsor over $\Gamma$, and to observe that
 the volume of $$(R^{\ast})^n=\spe_{\mathbb{G}^n_{m,R}}^{-1}(\mathbb{G}_{m,k}^n)$$ is $[\mathbb{G}^n_{m,k}]=(\LL-1)^n$.
\end{remark}

 To give an idea about the information contained in the motivic volume, let us explain how its realizations compare to more classical invariants.
  Let $X$ be an algebraic $K$-variety and let $S$ be a semialgebraic subset of $X$. If we denote by $X^{\an}$ the Berkovich analytification of $X$
  over the completion of $K$, then we can associate to $S$ a subset $S^{\an}$ of $X^{\an}$ in a canonical way, defined by the same formulas as $S$.
  If $S^{\an}$ is locally closed in $X^{\an}$, then the germ $(X^{\an},S^{\an})$ has finite $\ell$-adic cohomology \cite{M}, and we deduced from results of Hrushovski and Loeser
 \cite{HL} that the $\ell$-adic  Euler characteristic of $(X^{\an},S^{\an})$ is equal to the image of $\Vol(S)$ under the Euler characteristic realization $K_0(\Var_k)\to \Z$.
  Moreover, if $K$ is the field of complex Puiseux series and $X$ is an algebraic variety over $K$, then the Hodge-Deligne polynomial of $\Vol(X(K))$ equals the Hodge-Deligne polynomial of the limit mixed Hodge structure associated with $X$.  See \cite{NPS} for details.

\subsection{Semistable models}\label{ss:ss}
 Let us look at a class of examples where the motivic volume can easily be computed.
We say that a flat $R$-scheme of finite type $\cX$ is {\em strictly semistable} if it can be covered with open subschemes that admit an \'etale morphism to an $R$-scheme of the form
 $$\Spec R[x_0,\ldots,x_d]/(x_0\cdot \ldots \cdot x_r-a)$$
where $r\leq d$ and $a$ is a non-zero element of the maximal ideal of $R$.

 Let $\cX$ be a strictly semistable $R$-scheme of pure relative dimension $d$, and let $E_i,i\in I$ be the irreducible components of $\cX_k$. For every non-empty subset $J$ of $I$, we set $$E_J=\bigcap_{j\in J}E_j, \quad E_J^o=E_J\setminus \left(\bigcup_{i\notin J}E_i\right).$$ The subsets $E_J^o$ form a partition of $\cX_k$ into locally closed subsets.
 Decomposing $\cX(R)$ into the semialgebraic pieces $\spe_{\cX}^{-1}(E_J^o)$, one can show that
 $$[\cX(R)]=\sum_{\emptyset\neq J\subset I}(-1)^{|J|-1}\Theta(E_J^o,d-|J|+1)\cdot \Theta(\Delta_0,|J|-1)$$ where $\Delta_0$ is the $0$-simplex (as an intermediate step, one uses
 the scissor relations in a suitable Grothendieck ring of $G$-rational polyhedra to show that $\Theta(\Delta^{\circ},n)=(-1)^n\Theta(\Delta_0,n)$ for every $n>0$ and every open $n$-dimensional simplex $\Delta^{\circ}$ -- see Example \ref{ex:ss} for a special case of this calculation).
  In particular,
 \begin{equation}\label{eq:ssvol}\Vol(\cX(R))=\sum_{\emptyset \neq J\subset I}[E_J^o](1-\LL)^{|J|-1}.\end{equation}
 This implies that, when $\cX$ is defined over a formal power series ring, the motivic volume  $\Vol(\cX(R))$ coincides with Denef and Loeser's motivic nearby fiber of $\cX$ (see \cite{NPS} for a precise statement).

If $X$ is a smooth and proper $K$-variety, a strictly semistable model of $X$ is a strictly semi-stable proper $R$-scheme $\cX$ endowed with an isomorphism of $K$-schemes $\cX_K\to X$. If $K$ is the field of complex Puiseux series then such a strictly semistable model always exists, since $X$ is defined over a Laurent series field $K_0\subset K$ and we can apply resolution of singularities and the semistable reduction theorem over the (discrete) valuation ring of $K_0$. For general $K$, the existence of semistable models is not known. If $\cX$ is a strictly semistable model of $X$, then formula \eqref{eq:ssvol} becomes
$$\Vol(X(K))=\sum_{\emptyset \neq J\subset I}[E_J^o](1-\LL)^{|J|-1}.$$

\begin{example}\label{ex:ss}
Consider the $R$-scheme $$\cX=\Spec\,R[x,y]/(xy-a)$$ where $a$ is any nonzero element in the maximal ideal of $R$.
 Denote by $E_1$ the zero locus of $x$ and by $E_2$ the zero locus of $y$ in $\cX_k$. Then $E_{\{1,2\}}$ is the origin $O=(0,0)$ of $\cX_k$, and
   the locally closed subsets $E_1^o,\,E_2^o$ and $\{O\}$ form a partition of $\cX_k$. Thus $S_1=\spe_{\cX}^{-1}(E^o_1)$, $S_2=\spe_{\cX}^{-1}(E^o_2)$ and $S_{\{1,2\}}=\spe_{\cX}^{-1}(O)$ form a partition of $\cX(R)$ into semialgebraic subsets, and
   $$[\cX(R)]=[S_1]+[S_2]+[S_{\{1,2\}}]$$ in $K_0(\VF_K)$ by the scissor relations. Since $\cX$ is smooth over $R$ at every point of $E_1^o$ and $E_2^o$, we have
   $[S_1]=\Theta(E_1^o,1)$ and $[S_2]=\Theta(E_2^o,1)$ in $K_0(\VF_K)$.

   In order to describe the class of $S_{\{1,2\}}$ in $K_0(\VF_K)$, we observe that
   projection onto the $x$-coordinate defines a semialgebraic bijection between $S_{\{1,2\}}$ and the set $$\{x\in K^{\ast}\,|\,0<v(x)<v(a)\}=\trop^{-1}(\Gamma)$$
   with $\Gamma$ the open interval $(0,a)$ in $\R$. Thus $[S_{\{1,2\}}]=\Theta(\Gamma,1)$ in $K_0(\VF_K)$. We can further simplify this expression by
   noting that multiplication with $a$ defines a semialgebraic bijection between $\trop^{-1}(\R_{\geq 0})$ and $\trop^{-1}(\R_{\geq v(a)})$, so that
   $$[\trop^{-1}([0,a))]=[\trop^{-1}(\R_{\geq 0})]-[\trop^{-1}(\R_{\geq v(a)})]=0$$ in $K_0(\VF_K)$, and
   $$\Theta(\Gamma,1)=[\trop^{-1}(\Gamma)]=[\trop^{-1}([0,a))] - [\trop^{-1}(0)]=-\Theta(\Delta_0,1).$$ Adding up all the contributions, we conclude that
   $$[\cX(R)]=\Theta(E_1^o,1)+\Theta(E_2^o,1)-\Theta(\Delta_0,1)=\Theta(\Delta_0,1)$$ in $K_0(\VF_K)$, where the last equality follows from the fact that
   $E_1^o$ and $E_2^o$ are isomorphic to $\mathbb{G}_{m,k}$, and $\Theta(\mathbb{G}_{m,k},1)=\Theta(\Delta_0,1)$ by Example \ref{exam:rel}.

 In this particular example, we can perform the same calculation more efficiently by observing that projection onto the $x$-coordinate also defines a semialgebraic bijection between
   $\cX(R)$ and the set $$\{x\in K^{\ast}\,\vert\,0 \leq v(x) \leq v(a)\}=\trop^{-1}(\Gamma')$$ where $\Gamma'$ is the closed interval $[0,a]$ in $\R$.
   Since $\Gamma'$ is the disjoint union of $\Gamma$ and two $0$-simplices, this yields
   $$[\cX(R)]=\Theta(\Gamma',1)=\Theta(\Gamma,1)+2\Theta(\Delta_0,1)=\Theta(\Delta_0,1).$$
      For the motivic volume, we find
   $$\Vol(\cX(R))=[E_1^o]+[E_2^o]-(\LL-1)=(\LL-1)$$ in $K_0(\Var_k)$.
\end{example}

\section{Tropical computation of the motivic volume}
 In order to compute the motivic volume on a large and interesting class of examples, we have established an explicit formula for the motivic volume of sch\"{o}n subvarieties of algebraic tori in terms of their tropicalization. If $K$ is the field of complex Puiseux series, it follows from \cite[6.11]{LQ} that the classes of such varieties generate the Grothendieck group $K_0(\Var_K)$, so that this method can be used, in principle, to compute the motivic volume of any $K$-variety. This method tends to be simpler than finding strictly semistable models as in Section \ref{ss:ss}. A similar formula for Denef and Loeser's motivic nearby fiber was obtained (by means of a more involved argument) in \cite{KS}. Our approach yields more information because we also get an explicit description for the class of a sch{\"o}n variety in the Grothendieck ring of semialgebraic sets $K_0(\VF_K)$.

\subsection{Sch\"on varieties and tropical compactifications}\label{ss:tropical}
 Let $n$ be a positive integer and let $X$ be an integral closed subvariety of the algebraic torus $\mathbb{G}_{m,K}^n$. We assume that $X$ is {\em sch\"{o}n}, which is a standard non-degeneracy condition in tropical geometry that makes it possible to construct explicit compactifications of $X$ over the valuation ring $R$ with good properties, using toric geometry.
  This condition states that, for every element $a$ of $(K^{\ast})^n$, the schematic closure of $a^{-1}X$ in $\mathbb{G}_{m,R}^n$ is smooth over $R$.
  The special fiber of this schematic closure only depends on $w=\trop(a)\in \R^n$ (up to isomorphism of $k$-schemes) and is called the {\em initial degeneration} of $X$ at $w$.
  It is denoted by $\mathrm{in}_w(X)$.

 Let $\mathrm{Trop}(X)$ be the tropicalization of $X$, that is, the closure of the image of $X(K)$ under the tropicalization map
 $\trop:(K^{\ast})^n\to \R^n$. Let $\Sigma$ be a $G$-admissible tropical fan for $X$ in $\R^n\oplus \R_{\geq 0}$, in the sense of \cite[12.1]{gubler}.
  This is a fan whose rays are spanned by vectors in $G^{n}\oplus G_{\geq 0}$ and whose support is equal to the closure of the cone over $\mathrm{Trop}(X)\times \{1\}$ in
  $\R^n\oplus \R_{\geq 0}$. Intersecting the cones in $\Sigma$ with the affine subspace $\R^n\times \{1\}$ of $\R^{n+1}$, we obtain
a $G$-rational polyhedral subdivision of $\mathrm{Trop}(X)$, which we denote by $\Sigma_1$. On the other hand, by intersecting the cones of $\Sigma$ with the coordinate hyperplane $\R^n\times\{0\}$ in $\R^{n+1}$, we obtain a $G$-admissible fan in $\R^n$, which we denote by $\Sigma_0$ and which is called the {\em recession fan} of $\Sigma_1$.
   For every cell $\gamma$ in $\Sigma_1$, we set $\mathrm{in}_\gamma(X)=\mathrm{in}_w(X)$ where $w$ is any point in the relative interior $\mathring{\gamma}$ of $\gamma$ (in our terminology, cells are closed). This definition does not depend on the choice of $w$ (up to isomorphism of $k$-schemes).

 We denote by $\mathbb{P}(\Sigma)$ the toric $R$-scheme associated with $\Sigma$ \cite[\S7]{gubler}.
  This is an equivariant partial compactification of $\mathbb{G}_{m,K}^n$ over $R$, whose generic fiber is
  the toric variety over $K$ defined by the recession fan $\Sigma_0$, and whose special fiber is a union of toric varieties associated with the vertices of $\Sigma_1$.
   We denote by $\cX$ the schematic closure of $X$ in $\mathbb{P}(\Sigma)$.
   Then $\cX$ is proper over $R$ and the multiplication morphism
   $$m:\mathbb{G}^n_{m,R}\times_R \cX\to \mathbb{P}(\Sigma)$$ is faithfully flat, by the definition of a tropical fan. Moreover, the morphism $m$ is also smooth because of our assumption that $X$ is sch{\"o}n. Thus $m$ is smooth and surjective.

There exists a natural bijective correspondence between the set of cells in $\Sigma_1$ and the set of torus orbits in $\mathbb{P}(\Sigma)_k$, which is inclusion reversing on orbit closures. For every cell $\gamma$, we will denote the corresponding torus orbit by $O(\gamma)$, and we write $\cX_k(\gamma)$ for the intersection $O(\gamma)\cap \cX_k$ (with its reduced induced structure). This is a smooth subvariety of $\cX_k$. We set
$$X_\gamma= X(K)\cap \spe_{\cX}^{-1}(\cX_k(\gamma)).$$
This set is also equal to $X(K)\cap \trop^{-1}(\mathring{\gamma})$. As $\gamma$ ranges over the cells in $\Sigma_1$, the sets $X_\gamma$ form a semialgebraic partition of $X(K)$.

\subsection{A tropical formula for the motivic volume}
The following theorem gives an explicit formula for the class of a sch\"on subvariety of a torus in the Grothendieck ring of semialgebraic sets.

\begin{theorem}\label{thm:trop}
For every cell $\gamma$ in $\Sigma_1$, we have
$$[X_\gamma]=\Theta(\cX_k(\gamma),d-\mathrm{dim}(\gamma))\cdot \Theta(\mathring{\gamma},\mathrm{dim}(\gamma))$$ in $K_0(\VF_K)$.
Hence,
$$[X(K)]=\sum_{\gamma\in \Sigma_1}\Theta(\cX_k(\gamma),d-\mathrm{dim}(\gamma))\cdot \Theta(\mathring{\gamma},\mathrm{dim}(\gamma))$$ in $K_0(\VF_K)$.
\end{theorem}

The theorem is proven by constructing a semialgebraic bijection between $X_\gamma$ and a semialgebraic set of the form $\cY(R)\times \trop^{-1}(\mathring{\gamma'})$
 where $\cY$ is a smooth $R$-scheme with special fiber isomorphic to $\cX_k(\gamma)$ and $\gamma'$ is an embedding of the polyhedron $\gamma$ in $\R^{\mathrm{dim}(\gamma)}$.
 See \cite{NPS} for a detailed argument. As a consequence, we obtain the following expression for the motivic volume of $X(K)$.

 \begin{cor}
 We have
 $$\Vol(X(K))=\sum_{\gamma}(-1)^{\mathrm{dim}(\gamma)}[\cX_k(\gamma)](\LL-1)^{\mathrm{dim}(\gamma)}=\sum_{\gamma}(-1)^{\mathrm{dim}(\gamma)}[\mathrm{in}_{\gamma}(X)]$$ in $K_0(\Var_k)$, where $\gamma$ runs over the {\em bounded} cells in $\Sigma$.
 \end{cor}
 \begin{proof}
To prove the first equality, it suffices to observe that for every cell $\gamma$ of $\Sigma_1$, the additive invariant
 $\chi'(\mathring{\gamma})$ vanishes if $\gamma$ is unbounded, and equals $(-1)^{\mathrm{dim}(\gamma)}$ if $\gamma$ is bounded. The second equality follows from the fact that
 $\mathrm{in}_{\gamma}X$ is a $\mathbb{G}^{\mathrm{dim}(\gamma)}_{m,k}$-torsor over $\cX_k(\gamma)$.
 \end{proof}
We have proven similar formulas for the schematic closure of $X$ in the generic fiber of $\mathbb{P}(\Sigma)$.

\section{Application: refined Severi degrees}\label{sec:appli}
 Our main motivation for proving Theorem \ref{thm:trop} was to find a geometric interpretation for Block and G\"ottsche's refined tropical multiplicities \cite{BG}, which were introduced
 as the tropical counterparts of the refined Severi degrees of G\"ottsche and Shende \cite{GS}. We will briefly explain the general ideas.

\subsection{The refined Severi degrees of G{\"o}ttsche and Shende}
  Let $F$ be an algebraically closed field of characteristic zero.
   We denote by $\eu:K_0(\Var_F)\to \Z$ the ring morphism that sends the class of each $F$-scheme of finite type $X$ to the $\ell$-adic Euler characteristic of $X$, for any prime $\ell$.
  A {\em curve} over $F$ will mean a connected projective $F$-scheme of pure dimension one. We do not assume it to be reduced or irreducible.
  If $U$ is a Noetherian $F$-scheme, then a {\em family of curves} over $U$ is a flat projective morphism $\cC\to U$ whose geometric fibers are curves.
 We denote by $\Hilb^i_{\cC/U}$ the relative Hilbert scheme of $i$ points of the family $\cC\to U$.

 \begin{definition} Let $U$ be a connected $F$-scheme of finite type and let $\cC\to U$ be a family of curves over $U$.
 The motivic Hilbert zeta function of this family is the generating series
 $$Z_{\cC}(q)=\sum_{i\geq 0}[\Hilb^i_{\cC/U}]q^{i}$$ in $K_0(\Var_F)\llbr q \rrbr$.
 \end{definition}

 When $U=\Spec F$ and $\cC$ is smooth over $F$, then $\Hilb^i_{\cC/U}$ is isomorphic to the $i$-th symmetric power of $\cC$ and $Z_{\cC}(q)$ coincides with Kapranov's motivic zeta function \cite{kapranov}, a motivic upgrade of the
 Hasse-Weil zeta function for varieties over finite fields. Kapranov has proven that it is a rational function in $q$. More precisely, $(1-q)(1-q\LL)Z_{\cC}(q)$ is a polynomial of degree $2g$ in $K_0(\Var_F)[q]$, where $g$ denotes the genus of $\cC$. This result has been generalized to singular curves: see, for instance, Proposition 15 in \cite{GS} for the case of integral Gorenstein curves.

 In order to extract invariants from $Z_{\cC}(q)$, for general families $\cC\to U$, it is convenient to rearrange the terms in the generating series as in \cite[\S2.1]{GS}. If we denote by $g$ the arithmetic genus of the curves in the family, then there exists a unique sequence $N^{\mot}_0(\cC),N^{\mot}_1(\cC),\ldots$ of elements in $K_0(\Var_F)$ such that
 $$q^{1-g}Z_{\cC}(q)=\sum_{i=0}^{\infty}N_i^{\mot}(\cC)\left(\frac{q}{(1-q)(1-q\LL)}\right)^{i+1-g}.$$
  This is simply a formal consequence of the fact that the change of variable $q\mapsto q/(1-q)(1-q\LL)$ defines an automorphism of $K_0(\Var_F)\llbr q \rrbr$.
   Multiplying both sides with $q^{g-1}$ and setting $q=0$ reveals that $N_0^{\mot}(\cC)=[U]$.  For every $i\geq 0$, we set $n_i(\cC)=\eu(N^{\mot}_i(\cC))$. These invariants carry interesting enumerative information, as is illustrated by the following result.

 \begin{theorem}[Pandharipande-Thomas \cite{PT}]\label{thm:PT}
 Let $U$ be a connected $F$-scheme of finite type and let $\cC\to U$ be a family of reduced Gorenstein curves of arithmetic genus $g$ over $U$. Let $\delta$ be an element in $\{0,\ldots,g\}$.
 Assume that the family $\cC$ contains finitely many fibers of geometric genus $g-\delta$ and that these fibers have only nodal singularities (we will call such curves $\delta$-nodal).
 Assume moreover that the geometric genus of all the other fibers in $\cC$ is strictly larger than $g-\delta$. Then $n_i(\cC)$ vanishes for $i>\delta$, and
 $n_{\delta}(\cC)$ equals the number of $\delta$-nodal curves in $\cC$.
\end{theorem}
\begin{proof}
The proof relies in a crucial way on the integral calculus of Euler characteristics: let $f:Y\to X$ be a morphism of $F$-schemes of finite type.
For every integer $n$ we denote by $X_n$ the set of points $x\in X$ such that the fiber of $f$ over $x$ has Euler characteristic $n$.
 Then $X_n$ is empty for all but finitely many $n$, the sets $X_n$ form a partition of $X$ into constructible subsets, and $$\eu(Y)=\sum_{n\in \Z}\eu(X_n)\cdot n.$$
 We express this property by saying that we can compute the Euler characteristic of a family by integrating the Euler characteristics of the fibers over the base.
  It follows that we can also compute the invariants $n_i(\cC)$ by integration over the base, so that we may assume that the family $\cC$ consists of a single curve $C$ over $F$.
    Now the theorem is a direct consequence of the following result from Appendix B.1 in \cite{PT}: if the geometric genus of $C$ is strictly larger than $g-\delta$ then $n_i(C)$ vanishes for $i\geq \delta$. If $C$ is $\delta$-nodal then $n_i(C)$ vanishes for $i>\delta$ and $n_{\delta}(C)=1$.
    (Beware that our indexation of the invariants $n_i(\cC)$ is different from the one in \cite{PT}; we follow the convention in \cite{GS}.)
   \end{proof}
 The conditions in the theorem are satisfied in many interesting cases, for instance, for the universal curve of a general $\delta$-dimensional subspace of the linear system attached to a $\delta$-very ample line bundle on a smooth proper surface \cite[2.1]{KST}.

  Now it is natural to ask what kind of finer geometric information is contained in the motivic invariants  $N_i^{\mot}(\cC)$. Motivated by ideas from string theory,  G{\"o}ttsche and Shende have proposed in \cite{GS} to replace the Euler characteristic by the $\chi_{-y}$-genus. Recall that the $\chi_{-y}$-genus of a smooth and proper $F$-scheme $X$  is defined by
  $$\chi_{-y}(X) = \sum_{q} (-1)^{q}\chi(X,\Omega^q_{X/F})y^q.$$ It extends uniquely to a ring morphism
  $$\chi_{-y}:K_0(\Var_F)\to \Z[y]$$ so that we can define the $\chi_{-y}$-genus of an arbitrary $F$-variety by additivity.
    For instance, $$\chi_{-y}(\A^1_F)=\chi_{-y}(\mathbb{P}^1_F)-\chi_{-y}(\Spec F)=(y+1)-1=y.$$
  The $\chi_{-y}$-genus specializes to the Euler characteristic by setting $y=1$.
   Hence, the $\chi_{-y}$-genus of $N_i^{\mot}(\cC)$ can be viewed as a refinement of $n_i(\cC)$. We will denote it by $N_i(\cC)$; then $n_i(\cC)=N_i(\cC)\vert_{y=1}$.

\begin{remark}
One could further refine these invariants by upgrading the $\chi_{-y}$-genus by the Hodge-Deligne polynomial, or by working directly with the classes
$N_i^{\mot}(\cC)$. The problem with the invariants $N_i^{\mot}(\cC)$ is that, in the set-up of Section \ref{ss:BGmult} below, they might depend too strongly
 on the choice of the point configuration $S$. Their Hodge-Deligne realization will be independent of $S$ if this set is sufficiently general, but
 an important advantage of the $\chi_{-y}$-genus is that it has interesting vanishing properties: it annihilates every abelian variety of positive dimension.
  The effect on the invariants $N_i(\cC)$ is that they focus on the singularities in the family $\cC$. For instance, if $C$ is a smooth projective $F$-curve, then $N_0(C)=1$ and
  $N_i(C)=0$ for $i>0$. More generally, if $C$ is a proper integral Gorenstein curve over $F$, then $N_i(C)$ vanishes when $i$ is strictly larger than the cogenus of $C$ (the difference between the geometric and the arithmetic genus), by Corollary 23 in \cite{GS}. G\"ottsche and Shende conjecture that a similar vanishing result holds for suitable families of curves: see Conjecture 45 in \cite{GS}.
\end{remark}

  \subsection{The refined tropical multiplicities of Block and G{\"o}ttsche}\label{ss:BGmult}
Let $K$ be the field of complex Puiseux series. Let $\Delta$ be a lattice polygon in $\R^2$ with $n+1$ lattice points and $g$ interior lattice points. We denote by
$(X(\Delta),L(\Delta))$ the associated polarized toric surface over $K$.
 The complete linear series $|L(\Delta)|$ has dimension $n$, and its general member is a smooth projective curve of genus $g$.
  We fix an element $\delta$ in $\{0,\ldots,g\}$. Let
  $S$ be a set of $n-\delta$ closed points on the dense torus in $X(\Delta)$, and let $|L| \subset |L(\Delta)|$ be the linear series of curves passing through these points. We assume that
  the points in the tropicalization $\trop(S)\subset \R^2$  lie in tropical general position. We denote by $\cC\to |L|$ the universal curve over $|L|\cong \mathbb{P}_K^{g}$.

 Let $\Gamma$ be a tropical curve of genus $g-\delta$ and degree $\Delta$ through the points of $\trop(S)$. The {\em Mikhalkin multiplicity}
 $n(\Gamma)$ is a purely combinatorial invariant associated with $\Gamma$ \cite{mik}. The classical correspondence theorems in tropical geometry imply that $n(\Gamma)$ equals the number of integral $\delta$-nodal curves $C$ in $\cC$
  such that the intersection of $C$ with the dense torus in $X(\Delta)$ has tropicalization $\Gamma$ (henceforth, we will simply say that $C$ tropicalizes to $\Gamma$).
 In particular,
 if the family $\cC\to |L|$ satisfies the conditions in Theorem \ref{thm:PT} and all the $\delta$-nodal curves in the family are integral, then we can find the number $n_\delta(\cC)$ of $\delta$-nodal curves by solving a purely combinatorial problem, namely, counting the tropical curves $\Gamma$ with multiplicities $n(\Gamma)$.

  In \cite{BG}, Block and G{\"o}ttsche have defined refinements of the tropical multiplicities $n(\Gamma)$ to Laurent polynomials $N(\Gamma)$ in $\Z[y,y^{-1}]$ that
  specialize to $n(\Gamma)$ by setting $y=1$. It is expected that these refined multiplicities form the tropical counterpart of the refined invariants $N_\delta(\cC)$ (up to a renormalizing power of $y$); Block and G{\"o}ttsche have proven this for certain lattice polygons $\Delta$.
   The problem we address in \cite{NPS} is finding a geometric interpretation of the polynomial $N(\Gamma)$ for a fixed tropical curve $\Gamma$.
   As we have recalled above, the value $n(\Gamma)=N(\Gamma)|_{y=1}$ is equal to the number of $\delta$-nodal curves in $\cC$ that tropicalize to $\Gamma$.
      However, $N(\Gamma)$ is {\em not} simply the sum of the invariants $N_\delta(C)$ over the $\delta$-nodal curves $C$ in $\cC$ that tropicalize to $\Gamma$.
   The crucial complication is that we cannot compute the $\chi_{-y}$-genus of a family by integrating over the base, in general, except when the family is locally trivial in the Zariski topology. In fact, the finest invariant that is defined on the Grothendieck ring $K_0(\Var_k)$ and that can be computed on families by integrating over the base, is the Euler characteristic: such an invariant annihilates $\LL-1$ because $\mathbb{G}_{m,k}$ has an \'etale self-cover of degree $2$, and then induction on the dimension and Noether normalization easily imply that it must factor through the Euler characteristic.
        Let us look at a basic example to illustrate this problem.

   \begin{example}
   Consider the linear system of cubics through $8$ general points in $\mathbb{P}^2_K$. The universal family of this linear system is an elliptic pencil
   $\cC\to \mathbb{P}^1_K$ whose fibers are integral and have at worst nodal singularities. One can check that $N_1^{\mot}(\cC)=[\cC]$ (more generally, for a family of integral Gorenstein curves of arithmetic genus $g$ that admits a section, the invariant $N_g^{\mot}$ is the class of the relative compactified Jacobian, by the same reasoning as in Remark 18 of \cite{GS}).
      The total space $\cC$ is the blow-up of $\mathbb{P}^2_K$ at the $9$ base points of the linear system. Thus the Euler characteristic of $\cC$ is $12$, which implies that $\cC$ contains $12$ rational fibers, each of which has one node. The $\chi_{-y}$-genus of every smooth fiber equals $0$. Each rational fiber $C$ is isomorphic to $\mathbb{P}^1_K$ with two points identified, so that $$N_1(C)=\chi_{-y}(C)=\chi_{-y}(\mathbb{P}^1_K)-\chi_{-y}(\Spec K)=y.$$
       However, $$N_1(\cC)=\chi_{-y}(\cC)=\chi_{-y}(\mathbb{P}^2_K)+9\chi_{-y}(\A^1_K)=y^2+10y+1$$
    which is different from $12\chi_{-y}(C)=12y$. The reason is that, even though every smooth fiber in $\cC$ has $\chi_{-y}$-genus $0$,  the union of all the smooth fibers of $\cC\to \mathbb{P}^1_K$ has $\chi_{-y}$-genus $y^2-2y+1$.
   \end{example}

 The solution we propose is simple: instead of looking only at the $\delta$-nodal curves,
  we need to take {\em all} the curves in $\cC$ that tropicalize to $\Gamma$ into account. These form a semialgebraic set, and we can define invariants
 $N_i^{\mot}$, $N_i$ and $n_i$ as before by applying the motivic volume $\Vol$ to this semialgebraic set.
  Let us formulate this in a more precise way.  Let $\Gamma$ be a tropical curve of genus $g-\delta$ and degree $\Delta$ through the points of $\trop(S)$.  Let $|L|_{\Gamma}$ be the set of $K$-points in $|L|$ parameterizing curves in $\cC$ that tropicalize to $\Gamma$, and write $\cC_\Gamma$ for the preimage of $|L|_\Gamma$ in $\cC(K)$.
  For every $i\geq 0$ we denote by $\Hilb^i_{\cC_\Gamma}$ the preimage of $|L|_{\Gamma}$ in $\Hilb^i_{\cC/|L|}(K)$.
  We have seen in Example \ref{exam:semialg}\eqref{it:examtrop} that $|L|_\Gamma$, and thus $\Hilb^i_{\cC_\Gamma}$, are semialgebraic sets.
 We can define invariants $N_i^{\mot}(\cC_\Gamma)$ in $K_0(\Var_{\C})$ in the same way as before by means of the equality
$$\sum_{i\geq 0}\Vol(\Hilb^i_{\cC_\Gamma})q^{i+1-g}=\sum_{i=0}^{\infty}N_i^{\mot}(\cC_\Gamma)\left(\frac{q}{(1-q)(1-q\LL)}\right)^{i+1-g}$$ in $K_0(\Var_\C)\llbr q \rrbr$.
 Specializing with respect to the $\chi_{-y}$-genus, we again obtain polynomials $N_i(\cC_\Gamma)$ in $\Z[y]$.

   \begin{conj}
 Block and G{\"o}ttsche's refined tropical multiplicity $N(\Gamma)$ can be expressed as $N(\Gamma)=y^{-\delta}N_\delta(\cC_\Gamma)$.
      \end{conj}

 We have proven that this conjecture is correct after setting $y=1$, that is, the Mikhalkin multiplicity $n(\Gamma)$ is the Euler characteristic of $N^{\mot}_\delta(\cC_\Gamma)$. The proof makes use of Berkovich's $\ell$-adic cohomology for $K$-analytic spaces in order to show that the Euler characteristic of a semialgebraic family
 can still be computed by integrating over the base. We have also verified the conjecture in the case $g=1$, using our formula for the motivic volume of a sch{\"o}n variety (Theorem \ref{thm:trop}). We refer to \cite{NPS} for detailed arguments and additional background.

\subsection*{Acknowledgements} The results presented here are part of an ongoing project with Sam Payne and Franziska Schroeter \cite{NPS} and it is a pleasure to thank both of them for the friendly and interesting collaboration. I am grateful to the organizers of the AMS 2015 Summer Research Institute on Algebraic Geometry in Salt Lake City for the invitation to give a talk and to write this contribution for the proceedings. Finally, I would like to thank the referee for her or his careful reading of the text, and for making various suggestions to improve the presentation.

\bibliographystyle{amsplain}

\end{document}